\documentclass[11pt]{amsart}
\usepackage[margin=1.25in]{geometry}
\usepackage{amssymb}
\usepackage{verbatim}
\usepackage[usenames]{color}
\usepackage{hyperref}
\usepackage{url}
\usepackage[all]{xy}
\usepackage{tikz}
\usepackage{caption}

\newtheorem{theorem}{Theorem}[section]
\newtheorem{proposition}[theorem]{Proposition} 
 
\newtheorem{lemma}[theorem]{Lemma}

\theoremstyle{definition}
\newtheorem{definition}[theorem]{Definition}

\theoremstyle{remark}
\newtheorem{example}[theorem]{Example}

\newcommand{\cf}{\textnormal{cf}}

\newcommand{\n}{\ensuremath{\mathfrak{n}}}
\newcommand{\ran}{\mbox{ran}}
\newcommand{\surj}{\ensuremath{\preccurlyeq}}
\newcommand{\bisurj}{\ensuremath{\preccurlyeq\succcurlyeq}}

\title{Combinatorics of Reductions between Equivalence Relations}
\author{Dan Hathaway}
\author{Scott Schneider}
\address{Mathematics Department\\
University of Michigan\\
Ann Arbor, MI 48109--1043, U.S.A.}
\email{danhath@umich.edu}
\email{sms252@gmail.com}

\begin{document}

\begin{abstract}
We discuss combinatorial conditions for the existence of various types of reductions between equivalence relations, and in particular identify necessary and sufficient conditions for the existence of injective reductions.
\end{abstract}

\maketitle

\section{Introduction}

Let $E$ and $F$ be equivalence relations on sets $X$ and $Y$, respectively.  A \emph{homomorphism} from $E$ to $F$ is a function $\phi:X\to Y$ such that for all $x,x'\in X$, $x\mathrel{E}x'$ implies $\phi(x)\mathrel{F}\phi(x')$. A homomorphism $\phi$ from $E$ to $F$ induces a map $\tilde{\phi}:X/E\to Y/F$ between the quotients defined by $\tilde{\phi}([x]_E)=[\phi(x)]_F$. We obtain special kinds of homomorphisms by requiring $\phi$ or $\tilde{\phi}$ to have certain properties such as being one-to-one or onto. For instance if $\tilde{\phi}$ is one-to-one, then $\phi$ is called a \emph{reduction}.  In this note we study the combinatorics of reductions between equivalence relations, and attempt to identify necessary and sufficient conditions for the existence of reductions of various natural types.  We will see that certain types admit simple combinatorial characterizations while others do not.  Our main results are a necessary and sufficient condition for the existence of an injective reduction from $E$ to $F$ and a complete diagram of implications between the various types of reducibility that we consider. We work in the purely set-theoretic context without making any definability assumptions on equivalence relations or reductions.

Many of the combinatorial problems we consider may be viewed as special instances of the general matching problem addressed in \cite{ANS}. However, it is not easy to apply the abstract framework of \cite{ANS} to our context, and we give a comparatively simple proof of Theorem \ref{thm:main} below.

\section{Reductions of Equivalence Relations}

We now define the various types of homomorphisms that we will consider. Let $E$ and $F$ be equivalence relations on sets $X$ and $Y$, respectively, let $\phi:X\to Y$ be a homomorphism from $E$ to $F$, and let $\tilde{\phi}$ be the induced map on classes. We consider the following properties of the maps $\phi$ and $\tilde{\phi}$:
\begin{enumerate}
 \item[(i)] $\phi$ is one-to-one;
 \item[(ii)] $\phi$ is onto;
 \item[(iii)] $\tilde{\phi}$ is one-to-one;
 \item[(iv)] $\tilde{\phi}$ is onto;
 \item[(v)] $\ran(\phi)$ is \emph{$F$-invariant}; i.e., if $y\in\ran(\phi)$ and $y\mathrel{F}y'$ then $y'\in\ran(\phi)$.
\end{enumerate}
It is straightforward to check that the only implications holding between these properties are those following from the fact that $\phi$ is onto if and only if $\tilde{\phi}$ is onto and $\ran(\phi)$ is $F$-invariant. It follows that there are 16 distinct Boolean combinations of these properties. Since we will always take $\phi$ to be a reduction (i.e., we assume (iii) holds), this reduces the number of distinct combinations to 8.  We now introduce terminology and notation for these 8 types of reductions.

\begin{definition} Let $E$, $F$, $\phi$, and $\tilde{\phi}$ be as above.
\begin{enumerate}
 \item $\phi$ is a \emph{reduction} if (iii) holds;
 \item $\phi$ is an \emph{embedding} if (i) and (iii) hold;
 \item $\phi$ is a \emph{surjective reduction} if (ii) -- (v) hold;
 \item $\phi$ is an \emph{isomorphism} if (i) -- (v) hold
 \item $\phi$ is an \emph{invariant reduction} if (iii) and (v) hold;
 \item $\phi$ is a \emph{full reduction} if (iii) and (iv) hold;
 \item $\phi$ is an \emph{invariant embedding} if (i), (iii), and (v) hold;
 \item $\phi$ is a \emph{full embedding} if (i), (iii), and (iv) hold.
\end{enumerate}
\end{definition}

\begin{definition}
If $E$, $F$ are equivalence relations on sets $X$, $Y$, we say that $E$ is \emph{reducible} to $F$ and write $E\leq F$ if there is a reduction from $E$ to $F$, and we say that $E$ and $F$ are \emph{bireducible} and write $E\sim F$ if $E\leq F$ and $F\leq E$. We introduce analogous terminology and notation for the other types of reductions as follows:
\[
\hspace{-52mm}
\arraycolsep=8mm
\begin{array}{lcc}
\text{(1) \ \emph{reducible}} & \leq & \sim \\
\text{(2) \ \emph{embeddable}} & \sqsubseteq & \approx \\
\text{(3) \ \emph{surjectively reducible}} & \surj & \bisurj \\
\text{(4) \ \emph{isomorphic}} & \cong & \cong \\
\text{(5) \ \emph{invariantly reducible}} & \leq^i & \sim^i \\
\text{(6) \ \emph{fully reducible}} & \leq^f & \sim^f \\
\text{(7) \ \emph{invariantly embeddable}} & \sqsubseteq^i & \approx^i \\
\text{(8) \ \emph{fully embeddable}} & \sqsubseteq^f & \approx^f
\end{array}
\]
\end{definition}


We display all the direct implications between these relations in Figures \ref{FIG1} and \ref{FIG2}, and we include a proof of Proposition \ref{prop:complete} at the end of the paper.



\begin{proposition}\label{prop:complete} The diagrams in Figures \ref{FIG1} and \ref{FIG2} are complete; that is, in each diagram,  for every pair of nodes $A$ and $B$, the implication $A\Rightarrow B$ holds if and only if it is implied by the arrows in the diagram. 
\end{proposition}

\noindent Note, however, that certain implications involving more than two relations may not be displayed in the figures; for instance, the fact that $E\leq F\wedge F\leq E\Rightarrow E\leq^fF$ is not displayed in Figure \ref{FIG1}.

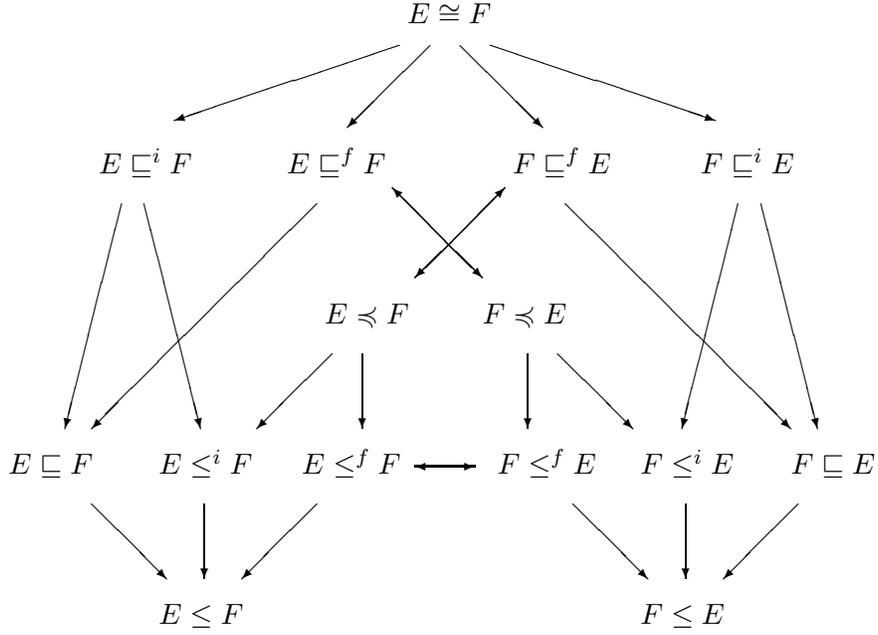
\begin{figure}
\setlength{\unitlength}{1mm}
\centering
\begin{picture}(12,90)

\put(1,80){$E\cong F$}

\put(-40,60){$E\sqsubseteq^i F$}
\put(-15,60){$E\sqsubseteq^f F$}
\put(15,60){$F\sqsubseteq^f E$}
\put(40,60){$F\sqsubseteq^i E$}

\put(-10,40){$E\preccurlyeq F$}
\put(11,40){$F\preccurlyeq E$}

\put(-52,20){$E\sqsubseteq F$}
\put(-32,20){$E\leq^i F$}
\put(-13,20){$E\leq^fF$}
\put(13,20){$F\leq^fE$}
\put(32,20){$F\leq^i E$}
\put(52,20){$F\sqsubseteq E$}

\put(-32,0){$E\leq F$}
\put(32,0){$F\leq E$}

\put(0,77){\vector(-3,-1){30}}
\put(4,77){\vector(-1,-1){11}}
\put(8,77){\vector(1,-1){11}}
\put(12,77){\vector(3,-1){30}}

\put(-1,58){\vector(1,-1){12}}
\put(11,46){\vector(-1,1){12}}
\put(14,58){\vector(-1,-1){12}}
\put(2,46){\vector(1,1){12}}

\put(-37,56){\vector(-1,-4){7.5}}
\put(-34,56){\vector(1,-4){7.5}}
\put(-11,56){\vector(-1,-1){30}}
\put(22,56){\vector(1,-1){30}}
\put(45,56){\vector(-1,-4){7.5}}
\put(48,56){\vector(1,-4){7.5}}

\put(-9,36){\vector(-1,-1){10}}
\put(-5,36){\vector(0,-1){10}}
\put(17,36){\vector(0,-1){10}}
\put(21,36){\vector(1,-1){10}}

\put(2,21){\vector(1,0){8}}
\put(10,21){\vector(-1,0){8}}

\put(-41,16){\vector(1,-1){10}}
\put(-26,16){\vector(0,-1){10}}
\put(-11,16){\vector(-1,-1){10}}
\put(23,16){\vector(1,-1){10}}
\put(38,16){\vector(0,-1){10}}
\put(53,16){\vector(-1,-1){10}}

\end{picture}
\caption{Implications between types of reducibility} \label{FIG1}
\end{figure}

\begin{figure}
\setlength{\unitlength}{1mm}
\centering
\begin{picture}(0,70)

\put(-5,61){\vector(1,0){10}}
\put(5,61){\vector(-1,0){10}}

\put(-15,57){\vector(0,-1){10}}
\put(15,57){\vector(0,-1){10}}

\put(-5,41){\vector(1,0){10}}
\put(5,41){\vector(-1,0){10}}

\put(-15,37){\vector(0,-1){10}}
\put(15,37){\vector(0,-1){10}}

\put(-15,17){\vector(0,-1){10}}
\put(15,17){\vector(0,-1){10}}

\put(-5,1){\vector(1,0){10}}
\put(5,1){\vector(-1,0){10}}

\put(-20,60){$E\cong F$}
\put(10,60){$E\approx^i F$}

\put(-23,40){$E \preccurlyeq\succcurlyeq F$}
\put(10,40){$E\approx^f F$}

\put(-20,20){$E\approx F$}
\put(10,20){$E\sim^i F$}

\put(-22,0){$E\sim^fF$}
\put(10,0){$E\sim F$}

\end{picture}
\caption{Implications between equivalences on the class of equivalence relations} \label{FIG2}
\end{figure}
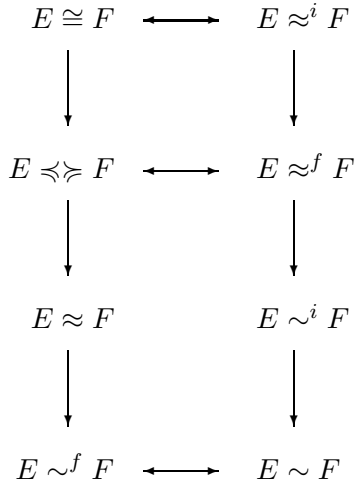

\section{The Main Theorem}

Now we consider the problem of finding necessary and sufficient combinatorial conditions for the existence of reductions of the various types between equivalence relations.

\begin{definition}
Given an equivalence relation $E$ and a (possibly finite) cardinal $\kappa$, let $\n_\kappa(E)$ be the number of $E$-classes of cardinality $\kappa$. Similarly, let $\n_{\geq\kappa}(E)$ be the number of $E$-classes of size at least $\kappa$ and $\n_{\leq\kappa}(E)$ the number of $E$-classes of size at most $\kappa$.
\end{definition}

\begin{theorem} Let $E$ and $F$ be equivalence relations on sets $X$ and $Y$, respectively. Then
\begin{enumerate}
 \item $E\leq F \ \Longleftrightarrow \ |X/E|\leq |Y/F|$;
 \item $E\sqsubseteq F \ \Longleftrightarrow \ (\forall\kappa) \, \n_{\geq\kappa}(E)\leq\n_{\geq\kappa}(F)$;
 \item $E\surj F \ \Longrightarrow \ (\forall\kappa)\, [\,\n_{\leq\kappa}(E)\leq\n_{\leq\kappa}(F) \; \wedge \; \n_{\geq\kappa}(E)\geq\n_{\geq\kappa}(F)\,]$;
 \item $E\cong F \ \Longleftrightarrow \ (\forall\kappa)\, \n_{\kappa}(E)=\n_{\kappa}(F)$;
 \item $E\leq^i F \ \Longrightarrow \ (\forall\kappa)\, \n_{\leq\kappa}(E)\leq\n_{\leq\kappa}(F)$;
 \item $E\leq^f F \ \Longleftrightarrow \ |X/E|=|Y/F|$;
 \item $E\sqsubseteq^i F \ \Longleftrightarrow \ (\forall\kappa)\, \n_{\kappa}(E)\leq\n_{\kappa}(F)$;
 \item $E\sqsubseteq^f F \ \Longleftrightarrow \ F\surj E$.
\end{enumerate}
\label{thm:main}
\end{theorem}

The bi-implications (1), (4), (6), and (7) are trivial to prove, as are the forward implications in (2), (3), and (5). The backward direction of (2) appears to be somewhat harder, and is our main result.  Additionally we will show that the necessary conditions given in (3) and (5) are not sufficient, and we argue that there are no simple combinatorial conditions characterizing the surjective or invariant reducibility of $E$ to $F$.

Now we present our proof of (2), which will make use of the following lemma.

\begin{lemma}\label{lem}
Let $\kappa$ be an infinite cardinal, and $A$ the class of ordinals that can be partitioned into $\kappa$ many cofinal subsets. Then $A$ is closed.
\end{lemma}

\begin{proof}
Let $\gamma$ be a limit point of $A$, and let $\langle \gamma_\alpha : \alpha < \cf( \gamma ) \rangle$ be a continuous increasing sequence of elements of $A$ with limit $\gamma$. For each $\alpha < \cf(\gamma)$, let $\{ P_\nu^\alpha : \nu < \kappa \}$ be a partition of $\gamma_\alpha$ into $\kappa$ many cofinal subsets. For each $\nu < \kappa$, define
\[
P_\nu := \bigcup_{\alpha < \cf(\gamma)} (P_\nu^{\alpha+1} - \gamma_\alpha).
\]
The set $\{ P_\nu : \nu < \kappa \}$ is a partition of $\gamma$ into $\kappa$ many cofinal subsets.
\end{proof}

Note that for an ordinal $\gamma$ and infinite cardinal $\kappa$, $\gamma$ may be partitioned into $\kappa$ many cofinal subsets iff $\gamma=\kappa\cdot\alpha$ for some ordinal $\alpha$.

\begin{proof}[Proof of Theorem \ref{thm:main} (2)]
The forward direction is clear. For the backward direction, we must show that there exists an injective function $\phi : X \to Y$ such that
\[
(\forall x, x' \in X)\; x \ E \ x' \Leftrightarrow \phi(x) \ F \ \phi(x')
\]
under the assumption that
\[
(\forall \kappa)\;\n_{\ge \kappa}(E) \, \le \, \n_{\ge \kappa}(F).
\]
Let us begin by fixing an enumeration $\langle C_\xi : \xi < \alpha \rangle$ of the $E$-classes such that for all $\xi < \eta < \alpha$, $|C_\xi| \le |C_\eta|$, as well as an enumeration $\langle D_\xi : \xi < \beta \rangle$ of the $F$-classes such that for all $\xi < \eta < \beta$, $|D_\xi| \le |D_\eta|$. Notice that since $\n_{\ge 1}(E) \le \n_{\ge 1}(F)$, we have $|\alpha| \leq |\beta|$.

It is not difficult to see that there exists an appropriate injection as long as $|\alpha|$ is finite, 
so for the remainder of the proof we assume $|\alpha|$ is infinite. Indeed, as an inductive hypothesis, assume we have proven the theorem for every pair of equivalence relations $(E', F')$ satisfying
\[
(\forall \kappa)\; \n_{\ge \kappa}(E') \, \le \, \n_{\ge \kappa}(F')
\]
such that the number of $E'$-classes is $< |\alpha|$.

Since $|\alpha| \le |\beta| \le \beta$, there is at least one ordinal $\gamma \le \beta$ that can be partitioned into $|\alpha|$ many cofinal subsets. By Lemma \ref{lem}, there is a largest such $\gamma \le \beta$, which we fix. We first claim that $|\beta - \gamma| < |\alpha|$. If not, let $\delta$ be the least ordinal such that $\gamma+\delta=\beta$, so that $|\delta|=|\beta-\gamma|$. Then 
\[
\gamma+|\alpha| \ \leq \ \gamma+|\beta-\gamma| \ = \ \gamma+|\delta| \ \leq \ \gamma+\delta \ = \ \beta,
\]
contradicting the choice of $\gamma$.


Let $\sigma < \alpha$ be the least ordinal such that $(\forall \xi < \gamma)\,|C_\sigma| > |D_\xi|$ if such an ordinal exists, and let $\sigma = \alpha$ otherwise.
Hence, for each $\nu < \sigma$ there is some $\xi' < \gamma$ such that $|C_\nu| \le |D_{\xi'}|$. Let $\{ P_\nu : \nu < \sigma \}$ be a partition of $\gamma$ into cofinal subsets (such a partition exists because $\gamma$ can be partitioned into $|\alpha|$ many cofinal subsets and $\sigma \le \alpha$).
Given any $\nu < \sigma$, we may pick a $\xi' < \gamma$ such that $|C_\nu| \le |D_{\xi'}|$, and then we may pick a $\xi \in P_\nu$ such that $\xi' \le \xi$ (so $|D_{\xi'}| \le |D_{\xi}|$). Hence,
\[
(\forall \nu < \sigma)(\exists \xi \in P_\nu)\; |C_\nu| \le |D_\xi|.
\]
Because of this, we may easily define an injection $\phi_1$ from $X_1 := \bigcup_{\nu < \sigma} C_\nu$ to $Y_1 := \bigcup_{\xi < \gamma} D_\xi$ such that
\[
(\forall x,x' \in X_1)\; x \ E \ x' \ \Leftrightarrow \ \phi_1(x) \ F \ \phi_1(x').
\]

If $\sigma = \alpha$ we are done, so assume $\sigma < \alpha$. Consider the sets 
\[
X_2 := \bigcup_{\sigma \le \nu < \alpha} C_\nu \quad \mbox{and} \quad Y_2 := \bigcup_{\gamma \le \xi < \beta} D_\xi.
\]
Let $E' := E \restriction X_2$ and $F' := F \restriction Y_2$. Since $|\beta - \gamma| < |\alpha|$, by the definition of $\sigma$ and the hypothesis that $\n_{\ge |C_\sigma|}(E) \le \n_{\ge |C_\sigma|}(F)$ we have that $|\alpha - \sigma| < |\alpha|$. That is, there are $< |\alpha|$ many $E'$-classes.
Also notice that $(\forall \kappa)\, \n_{\ge \kappa}(E') \le \n_{\ge \kappa}(F')$. We may now apply the inductive hypothesis to obtain an injective reduction $\phi_2$ from $E'\upharpoonright X_2$ to $F'\upharpoonright Y_2$. At this point we are finished, since 
\[
\phi\, :=\, \phi_1 \cup \phi_2
\]
 is an injective reduction from $E$ to $F$.
\end{proof}

\section{Counterexamples}

In this section we present some examples to show that the necessary conditions given in Theorem \ref{thm:main} for the existence of invariant and surjective reductions are not sufficient, and we argue that for these types of reducibility, no nice necessary and sufficient conditions exist.

\begin{example}\label{ex:1} Let $E$ and $F$ be equivalence relations each having exactly one equivalence class of size $n$ for each $1\leq n<\omega$ and no additional classes except that $E$ has exactly one class of size $\aleph_0$. Then for all cardinals $\kappa$ we have $\n_{\leq\kappa}(E)\leq\n_{\leq\kappa}(F)$ and $\n_{\geq\kappa}(E)\geq\n_{\geq\kappa}(F)$, but clearly there can be no invariant reduction from $E$ to $F$. \end{example}

To dispell the impression that finite cardinals are the sole source of the problem, we give another counterexample where this time $\n_\kappa(E)$ and $\n_\kappa(F)$ are either 0 or infinite for all $\kappa$. Our construction uses Fodor's Lemma, which is typical for the uncountable case of the matching problem (see, for instance, \cite[Lemma 4.9]{ANS}).

\begin{example}
There exist equivalence relations $E$ and $F$ such that
\begin{itemize}
\item[(1)] for all cardinals $\kappa$, $\n_\kappa(E)$ and $\n_\kappa(F)$ are either 0 or $\aleph_0$;
\item[(2)] $(\forall \kappa)\, \n_{\le \kappa}(E) = \n_{\le \kappa}(F)$;
\item[(3)] $(\forall \kappa)\, \n_{\ge \kappa}(E) = \n_{\ge \kappa}(F)$;
\item[(4)] $E\not\leq^iF$, and hence also $E\not\surj F$. 
\end{itemize} \label{ex:2}
\end{example}

\begin{proof}
It suffices to specify $\n_\kappa(E)$ and $\n_\kappa(F)$ for each cardinal $\kappa$. Let $\n_1(E) = \aleph_0$ and $\n_{\aleph_{\alpha}}(E) = \aleph_0$ for every limit ordinal $\alpha < \omega_1$, and let $\n_\kappa(E) = 0$ for every other cardinal $\kappa$. Let $\n_1(F) = \aleph_0$ and $\n_{\aleph_{\alpha+1}}(F) = \aleph_0$ for every limit ordinal $\alpha < \omega_1$, and let $\n_\kappa(F) = 0$ for every other cardinal $\kappa$.

It is clear that conditions (1) through (3) are satisfied. Suppose, towards a contradiction, that $\phi$ is an invariant reduction from $E$ to $F$. For every limit ordinal $\alpha < \omega_1$, $\phi$ maps each $E$-class of size $\aleph_{\alpha}$ onto an $F$-class of size $< \aleph_{\alpha}$. For each limit ordinal $\alpha < \omega_1$, arbitrarily pick some $E$-class $C_\alpha$ of size $\aleph_\alpha$. Hence, the function $\phi$ maps each class $C_\alpha$ onto some $F$-class of size $\aleph_{g(\alpha)}$ for some $g(\alpha) < \alpha$. We have now defined a regressive function $g$ from the (stationary) set of limit ordinals less than $\omega_1$ to $\omega_1$. By Fodor's Lemma, $g$ is constant on some stationary set. This means that there is some $\beta < \omega_1$ such that $\phi$ maps $\omega_1$ many $E$-classes onto $F$-classes of size $\aleph_\beta$. Since there are at most $\aleph_0$ many $F$-classes of size $\aleph_\beta$, we have a contradiction.
\end{proof}

Examples \ref{ex:1} and \ref{ex:2} suggest that in general there is no ``nice'' combinatorial characterization of the existence of an invariant or surjective reduction from one equivalence relation to another, and we now describe one way of making this precise. Define a \emph{nice condition} to be a conjunction of statements of the form ``for all cardinals $\kappa$, $a\mathrel{R}b$,'' where $a$ is one of the four terms
\[
\n_{\kappa}(E), \ \n_{\leq\kappa}(E), \ \n_{\geq\kappa}(E), \ |X/E|,
\]
$b$ is one of the four terms
\[
\n_{\kappa}(F), \ \n_{\leq\kappa}(F), \ \n_{\geq\kappa}(F), \ |Y/F|,
\]
and $R$ is one of the six relations
\[
\leq, \ \geq, \ =, \ \ne, \ <, \ >.
\]
The proof of the following proposition is tedious but not difficult, and we omit it.

\begin{proposition} Every nice condition which is implied by $E\leq^iF$ follows from the condition
\[
(\forall\kappa)\; \n_{\leq\kappa}(E) \, \leq \, \n_{\leq\kappa}(F),
\]
and every nice condition which is implied by $E\surj F$ follows from the condition 
\[
(\forall\kappa)\; [\,\n_{\leq\kappa}(E)\leq\n_{\leq\kappa}(F) \; \wedge \; \n_{\geq\kappa}(E)\geq\n_{\geq\kappa}(F)\,].
\]
\end{proposition}

\noindent In this sense parts (3) and (5) of Theorem \ref{thm:main} are optimal, and Examples \ref{ex:1} and \ref{ex:2} show that none of the relations $E\surj F$, $E\leq^iF$, and $E\sqsubseteq^fF$ can be characterized by a nice condition.

\section{Completeness of the Diagrams}

In this final section we prove Proposition \ref{prop:complete}.

\medskip

\noindent \emph{Proof that the diagram in Figure \ref{FIG1} is correct and complete}. 
All displayed implications follow immediately from the definitions, so we need only show that there are no additional implications. We will show that for every node $A$ in the diagram, there is no implication of the form $A\Rightarrow B$ that is not displayed. For the top node $E\cong F$ this is vacuous. By symmetry, it will suffice to consider the seven nodes on the left half of the diagram. We will accomplish this using the following seven pairs of equivalence relations, which are described as follows: $\langle n_1,\ldots,n_m\rangle$ denotes the equivalence relation having for each $1\leq k\leq m$ exactly $n_k$ equivalence classes of size $k$ and no others.

\begin{enumerate}
 \item $E=\langle 1\rangle$, $F=\langle 2\rangle$;
 \item $E=\langle 1\rangle$, $F=\langle 0,1\rangle$;
 \item $E=\langle 0,1\rangle$, $F=\langle 1\rangle$;
 \item $E=\langle 1\rangle$, $F=\langle 0,2\rangle$;
 \item $E=\langle 0,1\rangle$, $F=\langle 2\rangle$;
 \item $E=\langle 1,0,1\rangle$, $F=\langle 0,2\rangle$;
 \item $E=\langle 1,0,1\rangle$, $F=\langle 0,3\rangle$.
\end{enumerate}

\noindent (1) shows that $E\sqsubseteq^iF$ does not imply $F\leq E$. (2) shows that $E\sqsubseteq^fF$ implies neither $E\leq^iF$ nor $F\sqsubseteq E$. (3) shows that $E\surj F$ implies neither $E\sqsubseteq F$ nor $F\leq^iE$. (4) shows that $E\sqsubseteq F$ implies neither $E\leq^iF$ nor $F\leq E$. (5) shows that $E\leq^iF$ implies neither $E\sqsubseteq F$ nor $F\leq E$. (6) shows that $E\leq^fF$ implies none of $E\sqsubseteq F$, $E\leq^iF$, $F\sqsubseteq E$, and $F\leq^iE$. Finally, (7) shows that $E\leq F$ implies none of $E\sqsubseteq F$, $E\leq^iF$, and $F\leq E$. These observations suffice to establish the completeness of the diagram in Figure \ref{FIG1}.

\medskip

\noindent \emph{Proof that the diagram in Figure \ref{FIG2} is correct and complete}.
That $E\sim F\Rightarrow E\sim^fF$ is clear, and the fact that $E\approx^iF\Rightarrow E\cong F$ is well-known and follows from the standard Schr\"{o}der-Bernstein argument.  The remaining displayed implications follow immediately from the implications in Figure 1, so it is only left to show that there are no additional implications. For this it suffices to show the following:
\[
\begin{array}{llll}
(1) & E\bisurj F & \not\Rightarrow & E\cong F; \\
(2) & E\approx F & \not\Rightarrow & E\sim^iF; \\
(3) & E\sim^i F & \not\Rightarrow & E\approx F; \\
(4) & E\sim F & \not\Rightarrow & E\sim^iF; \\
(5) & E\sim F & \not\Rightarrow & E\approx F.
\end{array}
\]
This may be done using the following equivalence relations, which have no classes other than those described.
\begin{enumerate}
 \item $E$ has one class of size $n$ for each even integer $n$, $F$ has one class of size $n$ for each odd integer $n\geq 3$, and both $E$ and $F$ have $\aleph_0$ many classes of size 1.
 \item $E$ has $\aleph_0$ many classes of size $\aleph_0$ and one class of size 1; $F$ has $\aleph_0$ many classes of size $\aleph_0$ and one class of size 2.
 \item Both $E$ and $F$ have $\aleph_0$ many classes of size 1, and $E$ has one class of size 2.
 \item $E$ has one class of size 1, $F$ has one class of size 2.
 \item Same as (4).
\end{enumerate}

\end{document}